\documentclass[11pt,a4paper]{article}
\usepackage{mathrsfs, amsmath, amsthm, amssymb, bm, mathtools, thm-restate}
\usepackage{graphicx, xcolor, tikz}
\usetikzlibrary{calc, shapes}
\usepackage{caption, subcaption}
\usepackage{array}
\usepackage{cases}
\usepackage{tikz}
\usetikzlibrary{calc}
\usetikzlibrary{decorations.markings}
\usetikzlibrary{arrows}
\usetikzlibrary{patterns}
\usetikzlibrary{shapes,snakes}
\makeatletter
\makeatletter
\def\th@plain{%
	\upshape 
}
\makeatother

\makeatletter
\renewenvironment{proof}[1][\proofname]{\par
	\pushQED{\qed}%
	\normalfont \topsep6\p@\@plus6\p@\relax
	\trivlist
	\item[\hskip\labelsep
	\bfseries
	#1\@addpunct{.}]\ignorespaces
}{%
\popQED\endtrivlist\@endpefalse
}
\makeatother

\newtheorem{theorem}{Theorem}
\newtheorem{lem}{Lemma}
\newtheorem{cor}{Corollary}
\theoremstyle{definition}
\newtheorem{definition}{Definition}

\newtheorem{obs}{Observation}

\newtheorem{clm}{Claim}

\usepackage[left=25mm, top=25mm, bottom=25mm, right=25mm]{geometry}
\usepackage[T1]{fontenc}
\usepackage[inline]{enumitem}
\usepackage[square, numbers, sort&compress]{natbib}
\usepackage[pdftex,%
bookmarks=true,%
bookmarksnumbered=true, 
bookmarksopen=true, 
plainpages=false,%
pdfpagelabels,%
colorlinks=true, 
linkcolor=blue, 
citecolor=blue,%
anchorcolor=green,
urlcolor=black,
breaklinks=true,
hyperindex=true]{hyperref}

\def\int{\mathrm{int}}

\def \int {\rm int}

\usepackage[normalem]{ulem}
\begin{document}
	\title{Degree-truncated choice number of planar graphs }
	\author{Yiting Jiang$^1$\thanks{   Grant numbers:  NSFC 12301442, BK20230373, 23KJB110018.} \and Huijuan Xu$^2$ \and Xinbo Xu$^2$ \and Xuding Zhu$^2$\thanks{ Grant number:  NSFC 12371359.} }
\date{%
    $^1$School of Mathematical Sciences, Ministry of Education Key Laboratory of NSLSCS, Nanjing Normal University\\%
    $^2$School of Mathematical Sciences, Zhejiang Normal University\\[2ex]%
    \today
}
 
	\maketitle
	
	\begin{abstract}
	Assume $G$ is a graph and $k$ is a positive integer. Let $f:V(G)\to \mathbb{N}$ be defined as $f(v)=\min\{k,d_G(v)\}$. If $G$ is $f$-choosable, then we say $G$ is degree-truncated $k$-choosable. Answering a question of Richter, it was proved in [Zhou,Zhu,Zhu, Degree-truncated  choice number of graphs, arXiv:2308.15853] that there exists a 3-connected non-complete planar  graph that is not degree-truncated 7-choosable, and every 3-connected non-complete planar graph  is  degree-truncated 16-choosable. This paper improves the bounds, and proves that there exists a 3-connected non-complete planar graph that is not degree-truncated 8-choosable, and that every 3-connected non-complete planar graph is degree-truncated $12$-choosable.
 
\end{abstract}
\section{Introduction}

A \textit{list assignment} $L$ for a graph $G$ is a mapping that assigns a set $L(v)$ of permissible colours to each vertex $v$ of $G$.
We denote by $\mathbb{N}^G$   the set of mappings from $V(G)$ to $\mathbb{N}={\left\{0,1, \ldots \right\} }  $. 
For $f \in \mathbb{N}^G$, an $f$-list assignment is an assignment $L$ with $|L(v)| \geq f(v)$ for each vertex $v$.
An $L$-colouring of $G$ is a mapping $\phi$ that assigns a colour $\phi(v)$ from $L(v)$ to each vertex $v$, such that $\phi(u) \ne \phi(v)$ for every edge $uv$ of $G$. We say $G$ is \emph{ $f$-choosable} if $G$ is $L$-colourable for any $f$-list assignment $L$ of $G$, and say $G$ is \emph{$k$-choosable} if $G$ is $f$-choosable for the constant mapping $f \in \mathbb{N}^G$ defined as $f(v)=k$ for all $v$. The \textit{choice number $\operatorname{ch}(G)$} of $G$ is the minimum integer $k$ such that $G$ is $k$-choosable. We say   $G$ is   {\em  degree-choosable} if it is $f$-choosable for the mapping $f \in \mathbb{N}^G$ defined as $f(v)=d_G(v)$ for $v \in V(G)$, where $d_G(v)$ is the degree of vertex $v$. Degree-choosable graphs are characterized in \cite{ERT} and \cite{VI}: A connected graph $G$ is not degree-choosable if and only if $G$ is a Gallai-tree, i.e.,  each block of $G$ is either a complete graph or an odd cycle.

 In this paper, we investigate a combination of $k$-choosability and degree-choosability, which is called degree-truncated $k$-choosability in  \cite{ZZZ}.

\begin{definition}
	\label{def-tdc}
	Assume $G$ is a graph and $k$ is a positive integer. Let $f \in \mathbb{N}^G$ be defined as   $f(v)=\min\{k,d_G(v)\}$ for each vertex $v\in V(G)$. If  $G$ is $f$-choosable, then we say $G$ is \emph{degree-truncated $k$-choosable}. The {\em degree-truncated choice number } of $G$ is the minimum integer $k$ such that $G$ is degree-truncated $k$-choosable.
\end{definition}

In the degree-truncated $k$-list colouring model, each vertex $v$ has a list of size $k$, however, if a vertex $v$ has degree $d_G(v) < k$, then its list size is reduced to $d_G(v)$. In the study of $k$-list colouring of graphs, we do not need to consider vertices of degree less than $k$, as they can always be coloured properly, no matter what colours are assigned to its neighbours. In the degree-truncated $k$-list colouring model, vertices of small degree are also critical to colourability.  

If $G$ is a Gallai-tree, then $G$ is not degree-choosable, and hence not degree-truncated $k$-choosable for any $k$. 
The degree-truncated choice of $G$ is undefined. For any connected graph $G$, its degree-truncated choice number is at most $\Delta(G)+1$.

Degree-truncated choice number of graphs was first studied by Hutchinson \cite{JH}. A classical result of Thomassen \cite{Thomassen} in list colouring says that the choice number of every planar graph is at most 5. 
A natural question is the maximum degree-truncated choice number of planar graphs.

 For any positive integer $k$, the complete bipartite graph $K_{2,k^2}$ is planar and not degree-truncated $k$-choosable. Some other planar graphs of connectivity at most 2 are easily seen to have arbitrarily large  degree-truncated choice numbers.  Bruce Richter asked whether every 3-connected non-complete planar graph is degree-truncated 6-choosable  (cf. \cite{JH}).
Motivated by this question, Hutchinson studied the degree-truncated choice number of outerplanar graphs.  
It was proved in \cite{JH} that if $G \ne K_3$ is a 2-connected maximal outerplanar graph (i.e., each inner face is a triangle), then $G$ is degree-truncated 5-choosable. If $G$ is a  2-connected  bipartite outerplanar graph, then $G$ is 4-truncated degree-choosable. The results are sharp, as there are 2-connected maximal outerplanar graphs that are not degree-truncated 4-choosable, and 2-connected bipartite outerplanar graphs that are not degree-truncated 3-choosable. Hutchinson's result was strengthened in \cite{LWZZ}, where it was proved that 2-connected $K_{2,4}$-minor free graphs other than cycles and complete graphs are degree-truncated 5-DP-colourable. In \cite{DC},   some families of $K_5$-minor free graphs are shown to be degree-truncated $k$-choosable for $k= 6,7,8$, by putting restrictions on the distance between components of the subgraph induced by vertices of degree less than $k$.

Recently, Zhou, Zhu, and Zhu  \cite{ZZZ} answered Richter's question in negative and constructed a 3-connected non-complete planar graph that is not degree-truncated 7-choosable. 
They also proved that every 3-connected non-complete planar graph is  degree-truncated $16$-DP-colourable as well as degree-truncated $16$-AT, which implies that these graphs are degree-truncated $16$-choosable, as well as online degree-truncated $16$-choosable.  Before this result, it was not known whether the degree-truncated choice number of 3-connected non-complete planar graphs is bounded by a constant.

In this paper, we prove the following results.

\begin{theorem}
	\label{Not8truncated}
	There is a 3-connected non-complete planar graph that is not degree-truncated 8-choosable.
\end{theorem}

\begin{theorem}
	\label{thm-main}
	Every  3-connected   non-complete planar graph  is degree-truncated 12-choosable.
\end{theorem}

   We remark that our result takes advantage of the list assignment and  does not apply to DP-colouring, on-line list colouring, or Alon-Tarsi orientation.

\section{Proof of Theorem \ref{Not8truncated}}

In this section, we present a 3-connected non-complete planar graph that is not degree-truncated 8-choosable.

 \begin{figure}[!htp]
	\begin{center}
		\includegraphics[scale=0.6]{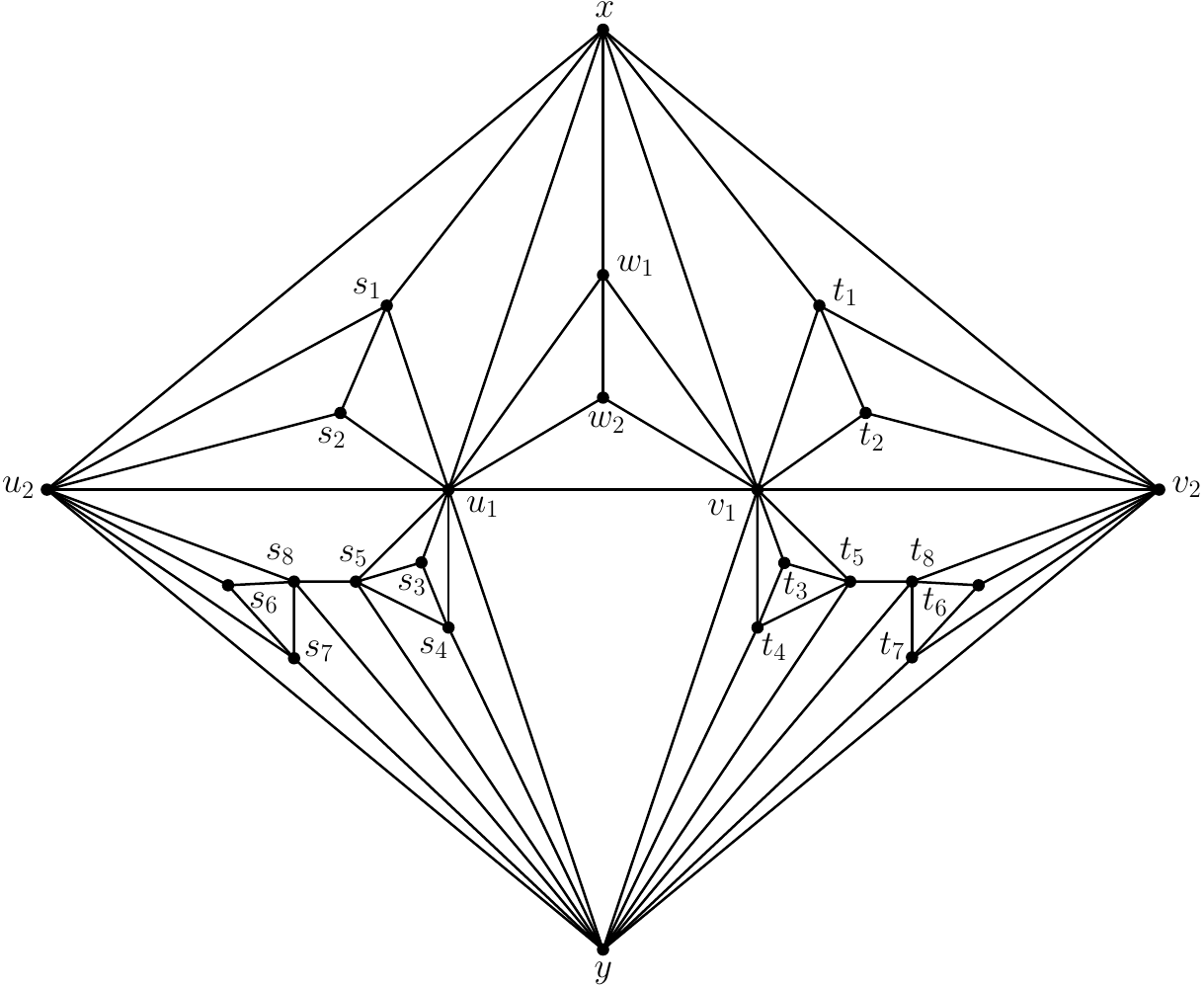}
	\end{center}
	\caption{The graph $H$}
	\label{fig_not8}
\end{figure}

  Let $H$ be the graph in Figure \ref{fig_not8} and 
 let $L$ be the list assignment of $H$  defined as follows:
 \begin{itemize}
 	\item $L(x)=\{a\}$, $L(y)=\{b\}$,
 	\item $L(u_i)=L(v_i)=\{a,b,1,2,3,4,5,6\}$ for $i=1,2$,
 	\item $L(s_1)=L(t_1)=\{a,1,2,3\}$, $L(w_1)=\{a,4,5,6\}$,
 	\item $L(s_2)=L(t_2)=L(s_3)=L(t_3)=\{1,2,3\}$, $L(w_2)=L(s_6)=L(t_6)=\{4,5,6\}$,
 	\item $L(s_4)=L(t_4)=\{b,1,2,3\}$, $L(s_7)=L(t_7)=\{b,4,5,6\}$,
 	\item $L(s_5)=L(t_5)=\{b,1,2,3,7\}$, $L(s_8)=L(t_8)=\{b,4,5,6,7\}$.
 \end{itemize}

First we show that $H$ is not $L$-colourable. Assume to the contrary that there is an $L$-colouring $\phi$ of $H$.

\begin{clm}
	$\phi(u_1)\in\{1,2,3\}$ or $\phi(v_1)\in\{1,2,3\}$.
\end{clm} 
\begin{proof}
	Let $H_1=H[\{u_1,v_1,w_1,w_2\}]$. 
	Then $\phi$ is an $L_1$-colouring of $H_1$ where $L_1(u_1)=L_1(v_1)=\{1,2,3,4,5,6\}$ and $L_1(w_1)=L_1(w_2)=\{4,5,6\}$. Assume $\phi(u_1)\notin\{1,2,3\}$ and $\phi(v_1)\notin\{1,2,3\}$. Then $\{\phi(u_1),\phi(v_1),\phi(w_1),\phi(w_2)\}\subseteq\{4,5,6\}$. But $\{u_1,v_1,w_1,w_2\}$ induces a copy of $K_4$, a contradiction.
\end{proof}

By symmetry, we assume that $\phi(u_1)\in\{1,2,3\}$. 

\begin{obs}
	\label{obsK4}
   Assume $K$ is a copy of $K_4$ with vertex set $\{x_1,x_2,x_3,x_4\}$, and $L'$ is a list-assignment of $K$ such that $L'(x_1)=L'(x_2)=L'(x_3)=\{a_1,a_2,a_3\}$ and $L'(x_4)=\{a_1,a_2,a_3\}\cup B$ (where $B$ is a non-empty colour set and $B\cap\{a_1,a_2,a_3\}=\emptyset$). Then in each $L'$-colouring $\varphi$ of $K$, $\varphi(x_4)\in B$.
\end{obs}

Let $H_2=H[u_1,u_2,s_1,s_2]$. Then $\phi$ is an $L_2$-colouring of $H_2$ where $L_2(u_1)=L_2(s_1)=L_2(s_2)=\{1,2,3\}$ and $L_2(u_2)=\{1,2,3,4,5,6\}$. By Observation \ref{obsK4}, $\phi(u_2)\in\{4,5,6\}$.

Let $H_3=H[u_1,s_3,s_4,s_5]$. Then $\phi$ is an $L_3$-colouring of $H_3$ where $L_3(u_1)=L_3(s_3)=L_3(s_4)=\{1,2,3\}$ and $L_3(s_5)=\{1,2,3,7\}$. By Observation \ref{obsK4}, $\phi(s_5)=7$. 

Similarly, let $H_4=H[u_2,s_6,s_7,s_8]$. Then $\phi$ is an $L_4$-colouring of $H_4$ where $L_4(u_2)=L_4(s_6)=L_4(s_7)=\{4,5,6\}$ and $L_4(s_8)=\{4,5,6,7\}$. By Observation \ref{obsK4}, $\phi(s_8)=7$.

Thus $\phi(s_5)=\phi(s_8)=7$, but $s_5s_8$ is an edge in $H$, a contradiction. Hence $H$ is not $L$-colourable.

Let $G$ be a graph obtained from the disjoint union of 56 copies  $H_i$ of $H$ by identifying all copies of $x$ into a single vertex (also named as $x$) and all the copies of $y$ into a single vertex (also named as $y$), and adding the edges $v_2^{(i)}u_2^{(i+1)}$ (where $v_2^{(i)}$ and $u_2^{(i)}$ are the copies of $u_2$ and $v_2$ in $H_i$) for $i=1,2,\dots,55$, and adding an edge connecting $x$ and $y$. Then $G$ is a non-complete $3$-connected planar graph.

Let $L(x)=L(y)=\{a,b,c,d,e,f,g,h\}$. There are 56 possible $L$-colourings $\phi$ of $x$ and $y$. Each such  colouring $\phi$ corresponds to one copy of $H$. We define the list assignment of the corresponding copy of $H$ as $L$ by replacing $a$ with $\phi(x)$ and replacing $b$ with $\phi(y)$. It is easy to verify that $L(v)=\min\{d(v),8\}$ for any $v\in V(G)$. As every possible $L$-colouring of $x$ and $y$ cannot be extended to an $L$-colouring of some copy of $H$, we conclude that $G$ is not $L$-colourable. Hence $G$ is not degree-truncated 8-choosable.

\section{Two  preliminary lemmas} 
The remainder of this paper is devoted to the proof of Theorem \ref{thm-main}.
In this section, we present two preliminary lemmas. 

 If $G$ is a Gallai-tree, and $L$ is a list assignment of $G$ with $|L(v)| \ge d_G(v)$ and $G$ is not $L$-colourable, then we say $L$ is a {\em bad list assignment} for $G$. The following was proved in \cite{ERT}.

\begin{lem} 
\label{lem-Gallai} 
If $G$ is a   Gallai-tree, and $L$ is a bad list assignment for $G$, then 
   for each block $B$ of $G$ that is $r$-regular, there is a set $C_B$ of $r$ colours such that (i) if $B$ and $B'$ share a vertex, then $C_B \cap C_{B'} = \emptyset$, and (ii)  for each vertex $v$, $L(v)= \cup_{v \in B}C_B$. 
\end{lem}

Note that each block of a Gallai-tree is a regular graph. It follows from Lemma \ref{lem-Gallai} that for a bad list assignment $L$ of a Gallai-tree $G$,   $|L(v)| =  d_{G}(v)$ for each vertex $v \in V(G)$. One property of a bad list assignment $L$ we shall use frequently is that if two vertices $u$ and $v$ are  in a same block   of $G$ and none of them is a   cut-vertex  of $G$, then $L(u)=L(v)$.

Assume $G$ is a simple plane graph. We denote by $F(G)$ the set of faces of $G$. 
For each face $\theta$ of $G$, let $V(\theta)$ be the set of vertices on the boundary of $\theta$ and $B(\theta)$ be the subgraph of $G$ induced by the boundary edges  of $\theta$. Let ${\rm int}(\theta)$ be the set of vertices of $G$ contained in $\theta$ (not including vertices on the boundary of $\theta$). Note that $G$ may be disconnected. Hence $B(\theta)$ also may be disconnected.

Let $\Theta(G)$ be the bipartite graph with partite sets $V(G)$ and $F(G)$, where $v\theta$ is an edge in $\Theta(G)$ if and only if $v \in V(\theta)$.

Let $\theta^*$ be the infinite face of $G$, and $v^*$ be a vertex on the boundary of $\theta^*$. A spanning subgraph $F$ of $\Theta(G)$ is called  {\em very nice} (with respect to $(\theta^*, v^*)$) if the following hold:

\begin{itemize}
  \item $d_F(v^*)=1$ and $d_F(v) \leq 2$ for  $v \in V(G) - \{v^*\}$.
  \item $d_F(\theta^*) = d_{\Theta(G)}(\theta^*)$ and 
  $d_F(\theta) = d_{\Theta(G)}(\theta)-2$ for each finite face $\theta$ of $G$.
    \item If $v_1, v_2 \in V(\theta)-N_F(\theta)$, then $v_1$ and $v_2$ are contained in a cycle $C$ of $V(\theta)$.
\end{itemize}

The following lemma was proved in \cite{ZZZ}.

\begin{lem}
    \label{lem-verynice}
    For any plane graph $G$, $\Theta(G)$ has a very nice subgraphs.
\end{lem}

If $G$ is 2-connected, then Lemma \ref{lem-verynice} follows from the result that $G$ has a \emph{ bipolar orientation}  (cf. \cite{RT1986}), which is an acyclic orientation of $G$ for which the following hold:
\begin{itemize}
    \item There is a unique source $s$, and a  unique sink $t$, and $s,t$ are two  arbitrary chosen vertices on the boundary of the infinite face of $G$. 
    \item For each face $\theta$, its boundary $B(\theta)$ consists of two directed paths (and hence the restriction of the orientation to $B(\theta)$  also has a unique source and a unique sink.
    \item For each vertex $v$, the in-edges (and hence the out-edges) are consecutive in the cyclic order of edges around $v$.
\end{itemize}

If $G$ is 2-connected and has a bipolar orientation, then let $F$   be the spanning subgraph of $\Theta(G)$ in which $v\theta$ is an edge if   $v\theta \in E(\Theta(G))$ and either $v$ is neither a source nor a sink of $B(\theta)$, or $\theta$ is the infinite face of $G$. Then $F$ is a very nice subgraph of $\Theta(G)$. 
 
 In case $G$ is connected but not 2-connected, we can construct a very nice subgraph of $\Theta(G)$ as follows: 

 Let $T$ be the tree whose vertices are blocks of $G$ and cut vertices of $G$, in which $vB$ is an edge if $B$ is block containing   $v$ (which is a cut vertex of $G$). Choose one block whose infinite face containing the infinite face of $G$ as the root of $T$, so that $T$ is a rooted tree. Each block $B$ other than the root block has a  unique father which is a cut vertex $v_B$ contained in $B$. By the choice of  the root block, we know that $v_B$ lies on the boundary of the infinite face $\theta_B$ of $B$.   
 Construct a very nice subgraph $F_B$ of $\Theta(B)$ with respect to $(\theta_B, v_B)$. Let $F$ be the union of the $F_B$'s for all the blocks $B$ of $G$. Then $F$ is a very nice subgraph of $\Theta(G)$.  

Note that when taking the union of the $  F_B$'s, faces of distinct blocks may correspond to a single face of $G$, and hence the vertices representing   faces 
of these distinct blocks are identified into a single vertex that represents a single  face of $G$. Note that in this case, at most one of these faces of the distinct blocks is a finite face of that block. Similarly, cut vertices from distinct blocks may correspond to a single cut vertex of $G$, and they are identified into a single vertex when taking the union. 

If $G$ is disconnected, then the union of very nice subgraphs of $\Theta(G_i)$ for the connected components $G_i$ of $G$ is a very nice subgraph of $\Theta(G)$.

 \section{A more technical statement}

We shall prove Theorem \ref{thm-main} by induction on the number of vertices of $G$. 
For that purpose, we need to conclude by induction hypothesis that certain induced subgraphs of $G$ can be properly coloured. However,   Theorem \ref{thm-main} refers to 3-connected planar graphs, and  the induced subgraphs we would like to colour by induction hypothesis may not be 3-connected. To overcome this difficulty, instead of Theorem \ref{thm-main}, we shall prove a stronger and more technical statement by induction.

Assume $G$ is a plane graph, $V_1, V_2$ is a partition of $V(G)$ such that each face of $G[V_2]$ contains at most one connected component of $G[V_1]$.

\begin{definition}
    For a connected component $Q$ of $G[V_1]$, let $\theta_Q$ be the face of $G[V_2]$ that contains $Q$.
\end{definition}
 As $G[V_2]$ may have more than one connected components, the subgraph of $G[V_2]$ induced by $V(\theta_Q)$   
could be disconnected.

Assume $C$ is a cycle in $G $.   We denote by ${\rm int}(C)$ (respectively, ${\rm ext}(C)$) the set vertices of $G$ contained in   the finite region  (respectively, the infinite region) with boundary $C$. Note that $V(C) \cup  {\rm int}(C)  \cup {\rm ext}(C)$ is a partition of $V(G)$.

\begin{definition}
  We say a connected component $Q$ of $G[V_1]$  is {\em properly connected to $V_2$} if the following hold:
\begin{enumerate}
    \item[(P1)] Every vertex in $V({\theta}_Q)$ is adjacent to some vertex of $Q$.
    \item[(P2)]  If $C$ is a cycle in $G[V(\theta_Q)]$ and $V(Q) \subseteq {\rm int}(C)$,   then for any vertex $v$ of $Q$, $G$ has three   paths contained in $V(C) \cup {\rm int}(C)$ connecting $V(C)$ and $v$, and these paths are vertex disjoint, except that they  share the same end vertex $v$. 
   \end{enumerate}
\end{definition}

 Instead of Theorem \ref{thm-main}, we shall prove the following more technical statement.
 
\begin{theorem}
    \label{thm-main2}
    Assume $G$ is a connected plane graph, and  $V_1, V_2$ is a partition of $V(G)$,  where $V_2 \ne \emptyset$ and each face of $G[V_2]$ contains at most one connected component of $G[V_1]$.   Assume   each connected component $Q$ of $G[V_1]$ is properly connected to $V_2$.  Assume $\theta^*$ is the infinite face of $G[V_2]$ and $v^* \in V(\theta^*)$. If $f: V(G) \to \mathbb{N}$ is defined as follows:
    \[
    f(v) = \begin{cases} 
    1, &\text{ if $v = v^*$ and $v^*$ is not an isolated vertex in $G[V_2]$}, \cr
    12, &\text{ if $v \in V_2 -\{v^*\}$ or $v=v^*$ is an isolated vertex in $G[V_2]$}, \cr 
    d_G(v), &\text{ if $v \in V_1$}. 
    \end{cases} 
    \]
    Then $G$ is $f$-choosable.
\end{theorem}

{\bf Proof of Theorem \ref{thm-main} (using Theorem \ref{thm-main2})} 
Assume $G$ is a counterexample to Theorem \ref{thm-main} with the minimum number of vertices and subject to this, with the maximum number of edges.   Let $V_1 = \{v \in V(G):d_G(v) \le 11\}$ and $V_2 = V(G)-V_1$.  Let $L$ be  a list assignment of $G$ such that $|L(v)| =d_G(v)$ for $v \in V_1$ and $|L(v)|=12$ for $v \in V_2$, and $G$ is not $L$-colourable. Since $G$ is 3-connected and not a complete graph, $G$ is not a Gallai-tree and hence $G$ is degree-choosable. Therefore, $V_2 \ne \emptyset$. 

 Two vertices $x$ and $y$ of $G$ are called {\em visible} to each other if $x$ and $y$ lie on the boundary of a same face of $G$, i.e., $G+xy$ is still a planar graph.

\begin{clm}\label{clm-visible}
    Any two vertices in $V_2$  visible to each other are adjacent. Consequently,   the following hold:
\begin{itemize}
    \item Each face of $G[V_2]$ contains at most one connected component of $G[V_1]$.
    \item For each connected component $Q$ of $G[V_1]$, each vertex in $V(\theta_Q)$ is adjacent to some vertex of $Q$.
\end{itemize}
\end{clm}
\begin{proof}
    If   $x,y \in V_2$ are visible to each other but not adjacent, then   $G+xy$ would not be $L$-colourable, and hence is a counterexample with more edges, contrary to the choice of $G$. 

Assume a face $\theta$ of $G[V_2]$ contains two connected components of $G[V_1]$. Since $G$ is 3-connected, each connected component of $G[V_1]$ is adjacent to at least 3 vertices of $V(\theta)$, we conclude that $V(\theta)$  has two non-adjacent vertices visible to each other, a contradiction. 

If $V(\theta_Q)$ has a vertex $v$ not   adjacent to any  vertex of $Q$, then the two neighbours of $v$ in $V(\theta)$ are visible to each other and no adjacent, a contradiction.
\end{proof}

Assume $Q$ is a connected component of $G[V_1]$, $v \in V(Q)$. By Claim \ref{clm-visible}, each vertex in $V(\theta_Q)$ is adjacent to some vertex of $Q$,  i.e., (P1) is satisfied.

Assume   $C$ is a cycle in $G[V(\theta_Q)]$ and $Q$ is contained in ${\rm int}(C)$. 
  Let $u$ be a vertex of $G$ that lies in ${\rm ext}(C)$ (or add such a vertex $u$ adjacent to all vertices in $C$ if ${\rm ext}(C) = \emptyset$). Since $G$ is 3-connected (if  ${\rm ext}(C) = \emptyset$, and    a vertex $u$ adjacent to all vertices of $C$ is added, then the resulting graph is also 3-connected), there are three internally vertex disjoint paths connecting $u$ and $v$. Each  of these paths  has a subpath   contained in  $V(C ) \cup {\rm int}(C)$ that connects $V(C)$ and $v$. These paths   are vertex disjoint,   except that they  have a common end vertex $v$.  Hence, (P2) is satisfied and  
$Q$ is properly connected to $V_2$.  

Let $f \in \mathbb{N}^G$ be defined as in Theorem 
\ref{thm-main2}. As $L$ is an $f$-list assignment of $G$, it follows from  Theorem \ref{thm-main2} that $G$ is $L$-colourable, a contradiction.

\section{Proof of   Theorem \ref{thm-main2}} 
Assume Theorem \ref{thm-main2} is not true, and $G$ is a counterexample with the minimum number of vertices. Let $f$ be the function defined in Theorem \ref{thm-main2}, $L$ be an $f$-list assignment of $G$ such that $G$ is not $L$-colourable. 
 
For  $X, Y \subseteq V(G)$, let $$N_Y(X)= N_G(X) \cap Y.$$

\begin{definition}
    Assume $X$ is a subset of $V(G)$ and $\phi$ is an $L$-colouring of $G[X]$. Let $L^\phi$ be the list assignment of $G-X$ defined as $$L^\phi(v)=L(v)-\{\phi(u):u\in N_X(v)\}$$ for each vertex $v\in V(G)-X$.
\end{definition}

\begin{definition}
    \label{def-ext}
    Assume $X \subseteq Y$ are subsets of $V(G)$, $\phi$ is an $L$-colouring of $G[X]$, and $\psi$ is an $L$-colouring of $G[Y]$. If $\psi(v) = \phi(v)$ for each vertex $v \in X$, then we say $\psi$ is an  extension of $\phi$ to $Y$.
\end{definition}

\begin{definition}
    \label{def-leafblock}
    A \textit{leaf block} of a connected graph $Q$ is a block $B$ of $Q$ that contains at most one cut-vertex of $Q$. If $B$ contains one cut-vertex $v$ of $Q$, then $v$ is called the {\em root} of $B$. The other vertices of $B$ are called {\em non-root}  vertices. If $Q$ is 2-connected, then  $Q$ itself is called a \textit{leaf block} and all vertices of $Q$ are non-root vertices. For a leaf-block $B$ of $Q$, we denote by $U_B$ the set of non-root vertices of $B$.
\end{definition}

\begin{lem}
\label{lem-connected}
Either $|V_2| =1$ or $G[V_2]$ is 2-connected.
\end{lem}
\begin{proof}
 Assume $|V_2| \ge 2$ and $G[V_2]$ is not 2-connected.

 \medskip
 \noindent
 {\bf Case 1}  There is a connected component ${\tilde{Q}}$ of $G[V_1]$ such that $V(\theta_{\tilde{Q}})$ contains a cycle $C$ for which ${\rm int}(C) \cap V_2 \ne \emptyset$. 
 
We choose such a   cycle $C$ so that $|{\rm int}(C)|$ is maximum.
  
 Let $G'=G-{\rm int} (C)$. 
 For $i=1,2$, let $V'_i= V_i \cap V(G')$. It follows from the construction that each face of $G'[V'_2]$ contains at most one connected component of $G'[V'_1]$.
  
 Now we show that each connected component $Q$ of $G'[V'_1]$ is properly connected to $V'_2$. 
 
Let $Q$ be a connected component   of $G'[V'_1]$. Then $Q$  is also a connected component of $G[V_1]$. Recall that $\theta_Q$ is the face of $G[V_2]$ containing $Q$. Let $\theta'_Q$ be the face of $G'[V'_2]$ containing $Q$. Then   $V(\theta'_Q) =V(\theta_Q) \cap V'_2  \subseteq  V(\theta_Q)$ and hence every vertex of $V(\theta'_Q)$ is adjacent to some vertex of $Q$, i.e., (P1) is satisfied.   

Next we show that (P2) is also satisfied by $G'$ and $Q$. 
  Let $C'$ be a cycle in $G'[V(\theta'_Q)]$ such that $Q$ is contained in ${\rm int}(C')$. Then $C'$ is also a cycle in $G[V(\theta_Q)]$. By our assumption,    for each vertex $v$ of $Q$, $G$ has three   paths, say $P_1,P_2,P_3$,  contained in $V(C') \cup {\rm int}(C')$ connecting $V(C')$ and $v$, and these paths are vertex disjoint, except that they  share the same end vertex $v$. 

If ${\rm int}(C')$ contains some vertices of ${\rm int} (C)$, then as ${\rm int} (C) $ is connected,  ${\rm int} (C) \subseteq {\rm int}(C')$. But this contradicts the choice of $C$. So ${\rm int}(C')$ contains no vertex of ${\rm int} (C)$. 
Hence  the three paths in $G$ connecting $V(C')$ to $v$ are also paths in $G'$, and $Q$ is properly connected to $G'[V'_2]$. 
 
 By the minimality of $G$, there is an $L$-colouring $\phi$ of $G'$. 

Let $H=G[{\rm int}(C)]$ and for $i=1,2$, let $V''_i = V_i \cap V(H)$. Then $V''_1 $ and $V''_2$ form a partition of $V(H)$, and $V''_2 \ne \emptyset$. It is follows from the construction that each face of $H[V''_2]$ contains at most one connected component of $H[V''_1]$.

If $V_2 \cap {\rm int}(C)$ is disconnected to $V(C)$ in $G[V_2]$, then let $u^*$ be an arbitrary vertex on the boundary of the infinite face of $H[V''_2]$. 
Otherwise,  since $C$ is a cycle contained in $V(\theta_{\tilde{Q}})$, there is a cut-vertex $u$ of $G[V_2]$ contained in $V(C)$, and $V_2 \cap {\rm int}(C)$ is disconnected to $V(C) - \{u\}$ in $G[V_2]$. In this case, let $u^*=u$. Note that in the latter case, $u^*$ is not an isolated vertex in $H[V''_2]$. 

Note that $|L^{\phi}(v)| =12$ for $v \in V''_2$, except that in the latter case above, $|L^{\phi}(u^*)| =1$. We shall show that every connected component $Q$ of $H[V''_1]$ is properly connected to $V''_2$ in $H$, and hence $H$ has an $L^{\phi}$-colouring $\psi$. The union $\phi \cup \psi$ is an $L$-colouring of $G$, which is  a contradiction.

Assume $Q$ is a connected component of $H[V''_1]$. Then it is also a connected component of $G[V_1]$. Recall that $\theta_Q$ is the face of $G[V_2]$ containing $Q$. Let $\theta'_Q$ be the face of $H[V''_2]$ containing $Q$. Then $V(\theta'_Q) = V(\theta_Q) \cap V''_2$. Hence every vertex of $V(\theta'_Q)$ is adjacent to some vertex of $Q$. 

Assume $C'$ is a cycle in $H[V(\theta'_Q)]$ and $Q \subseteq {\rm int}(C')$. Since  $G-V(H)$   is contained in ${\rm ext}(C')$,   we conclude that   for any vertex $v$ of $Q$, there are three paths in $H[ V(C') \cup {\rm int}(C')]$ connecting $V(C')$ to $v$, and the paths are vertex disjoint, except that they have the same end vertex $v$. So (P2) is satisfied by $H$ and $Q$. Therefore 
$H$ has a $L^{\phi}$-colouring $\psi$. The union of $\phi$ and $\psi$ is an $L$-colouring of $G$. This completes the proof of Case 1.

\medskip
 \noindent
 {\bf Case 2}  For any connected component $Q$ of $G[V_1]$, for any cycle  $C$ contained in  $V(\theta_Q)$,   ${\rm int}(C) \cap V_2 = \emptyset$.

Let $B$ be a connected component of $G[V_2]$   not containing $v^*$, or a leaf-block of $G[V_2]$ such that $v^* \notin U_B$. 

Let $\tau$ be the face of $G[V_2] - B$  containing $B$ (or the face of $G[V_2]  - U_B$  containing $U_B$,   if $B$ is leaf-block of $G[V_2]$). Let   $G'=G- {\rm int}(\tau)$. For $i=1,2$, let $V'_i = V_i \cap V(G')$. It is obvious that each face of  $G'[V'_2]$ contains at most one connected component of $G'[V'_1]$.

Assume $Q$ is a  connected component of $G'[V'_1]$. Then $Q$ is a connected component of $G[V_1]$. Recall that $\theta_Q$ is the face of $G[V_2]$ containing $Q$. Let $\theta'_Q$ be the face of $G'[V'_2]$ containing $Q$. As $V(\theta'_Q) \subseteq V(\theta_Q)$, we know that every vertex of $V(\theta'_Q)$ is adjacent to some vertex of $Q$. Assume $C$ is a cycle in $\theta'_Q$ such that $Q \subseteq {\rm int}(C)$. Then  $C \subseteq V(\theta_Q)$.  As $Q$ is properly connected to $V_2$ in $G$, for any vertex $v \in V(Q)$, there are three paths   in $V(C) \cup {\rm int}(C)$ connecting $V(C)$ to $v$, and these paths are vertex disjoint, except they have the same end vertex $v$.  By our assumption, 
${\rm int}(C) \cap V_2 = \emptyset$. Hence ${\rm int}(C) \cap {\rm int}(\tau) = \emptyset$. So these three paths are also paths in $G'$, and  $Q$   is properly connected to $V'_2$. By the minimality of $G$,   there is a proper $L$-colouring $\phi$ of $G'$. 

Let $H= G[{\rm int}(\tau)] = G - V(G')$. For $i=1,2$, let $V''_i=V_i \cap V(H)$. If $B$ is a leaf-block of $G[V_2]$, then let $u^*$ be the cut-vertex of $G[V_2]$ contained in $B$. Otherwise, let $u^*$ be an arbitrary vertex on the boundary of the infinite face of $H$. Then 
$|L^{\phi}(v)| =12$ for each vertex of $V(H[V''_2]) $, except that if $B$ is a leaf-block of $G[V_2]$, then 
 $|L^{\phi}(u^*)| =1$. Note that if $|L^{\phi}(u^*)| =1$, then  $u^*$ is not an isolated vertex in $H[V''_2]$. 
 
 It is obvious that each face of $H[V''_2]$ contains at most one connected component of $H[V''_1]$.   It is also easy to see that every connected component $Q$ of $H[V''_1]$ is properly connected to $V''_2$ in $H$. Hence $H$ has an $L^{\phi}$-colouring $\psi$. The union $\phi \cup \psi$ is an $L$-colouring of $G$, which is  a contradiction. This completes the proof of Case 2.
\end{proof}

\begin{lem}
    \label{lem-blocksofQ}
    Each connected component $Q$ of $G[V_1]$ is a Gallai-tree. Moreover,  the following hold:
    \begin{itemize}
        \item Each block of $Q$ is either $K_n$ for $n \le 3$ or is an odd cycle.
        \item For each non-cut vertex $v$ of $Q$, $d_Q(v) \le 2$.  If $\theta_Q$ is a finite face of $G[V_2]$, then $v$ is adjacent to some vertex in $V_2$.
    \end{itemize}   
\end{lem}
\begin{proof}
    If a connected component $Q$ of $G[V_1]$ is not a Gallai-tree, then by the minimality of $G$, $G-Q$ has a proper $L$-colouring $\phi$ (it is obvious that every connected component of $G[V_1]-Q$ is properly connected to $V_2$). As $|L^{\phi}(v)| \ge d_Q(v)$ for each vertex $v$ of $Q$ and $Q$ is not a Gallai-tree,  $Q$ admits a proper $L^{\phi}$-colouring $\psi$. The union $\phi \cup \psi$ is a proper $L$-colouring of $G$, a contradiction.

    Assume $Q$ is a connected component of $G[V_1]$ and $B$ is a block of $Q$. As $Q$ is a Gallai-tree, $B$ is either a complete graph or an odd cycle. If $B$ is neither an odd cycle nor $K_n$ for $n \le 3$, then $B$ is a copy of $K_4$. Then $B$ contains a triangle $T$ so that one vertex $v$ of $K_4$ is contained in the interior of $T$, and all the vertices in the interior of $T$ are not adjacent to $V_2$. Let $X$ be the set of vertices contained in  the interior of $T$, and let  
$\phi$ be a proper $L$-colouring of $X$ (which exists because $|L(v)| > d_{G[X]}(v)$ and $|L(x)| > d_{G[X]}(x)$ for every other vertex $x \in X$).
Then \[
    |L^{\phi}(v)|  \ge  \begin{cases} 
    1, &\text{ if $v=v^* $}, \cr
    12, &\text{ if $v\in V_2 $}, \cr 
    d_{G-X}(v), &\text{ if $v\in V_1-X$}, 
    \end{cases} 
    \]
   It is easy to see that   $Q-X$   is properly connected to $V_2$.  By the minimality of $G$, $G-X$ has a proper $L$-colouring $\psi$ and the union $\phi \cup \psi$ is a proper $L$-colouring of $G$, a contradiction.

    Assume $v \in Q$ is a non-cut vertex of $Q$. Since each block of $Q$ is an odd cycle or a complete graph of order at most $3$,   $d_Q(v) \le 2$. Since $Q$ is properly connected to $V_2$, $d_G(v) \ge 3$. Hence $v$ is adjacent to at least one vertex in $V_2$. 
\end{proof}

\begin{cor}
 \label{cor-C=theta}
Assume $Q$ is a connected component of $G[V_1]$. If $\theta_Q$ is a finite face of $G[V_2]$, then $\theta_Q=C_Q$ is a cycle, and for each leaf-block $B$ of $Q$, there are at least two vertices in $C_Q$ adjacent to $U_B$.  
\end{cor}
\begin{proof}
  Since $G[V_2]$ is 2-connected, the boundary of each face of $G[V_2]$ is a cycle. Thus, $\theta_Q=C_Q$.  Let $B$ be a leaf-block of $Q$ and $v \in U_B$. As   $\theta_Q$ is a finite face of $G[V_2]$ and $Q$ is properly connected to $V_2$, there exist three paths connecting $C_Q$ to $v$, and  these paths are pairwise vertex disjoint except that they share the same end vertex $v$. At most one of these paths contains the root vertex of $B$. Each of the other two paths contains an edge between $C_Q$ and $U_B$.  Therefore there are at least two vertices in $C_Q$ adjacent to $U_B$.  
\end{proof}

\begin{definition}
    \label{def-valid}
    A subset $X$ of $V(G)$ is {\em valid} if for each connected component $Q$ of $G[V_1]$, $Q-X$ is connected, and for each vertex $v \in X \cap V(Q)$, $N_{V_2}(v)   \subseteq X$.
\end{definition}

In the following proof process, we shall colour vertices of $G$ one by one, and the set $X$ of coloured vertices will always be a {\em valid} subset of $V(G)$. In other words, when we colour a vertex $v$ of a connected component $Q$ of $G[V_1]$, and $X$ is the set of already coloured vertices, then $v$ is not a cut-vertex of $Q-X$ and $v$ is not adjacent to any uncoloured vertex of $V_2$.

\begin{obs}
\label{obs-1}
    By following the above rules, if $\phi$ is an $L$-colouring of $G[X]$, then for any connected component $Q$ of $G[V_1]$, the following conditions hold: 
 \begin{enumerate}
     \item $Q-X$ is a connected component of $G[V_1-X]$.
     \item For any vertex $v \in V(Q)-X$, $|L^{\phi}(v)| \ge d_{G-X}(v)$.
     \item  For any vertex $v \in V_2-X$,  $|L^{\phi}(v)| \ge 12- |N_{V_2}(v) \cap X |$.
 \end{enumerate}  
\end{obs} 

 \begin{definition}
     \label{def-free}
     Assume $\phi$ is an $L$-colouring of $G[X]$, and $Q$ is a connected component of $G[V_1]$. We say $Q$ is {\em free} with respect to $\phi$ if for any extension $\psi$ of $\phi$ to $X \cup V_2$,  $Q-X$ 
 is $L^{\psi}$-colourable.
 \end{definition}

 Given an $L$-colouring $\phi$ of $G[X]$, it is non-trivial to check whether   $Q$   is free with respect to $\phi$. The following are two sufficient conditions for $Q$ to be free with respect to $\phi$.

\begin{lem}
    \label{lem-suff}
    Assume $\phi$ is an $L$-colouring   of $G[X]$ and $Q$ is a connected component of $G[V_1]$. Then $Q$ is free with respect to $\phi$ if one of the following holds:
\begin{enumerate}
    \item[(P1)] There exists a vertex $v \in V(Q)-X$ for which $|L^{\phi }(v)| > d_{G-X}(v)$.
    \item[(P2)] There are two non-cut vertices $u,v$ of a same block of $Q-X$ with $L^{\phi}(u) \ne L^{\phi}(v)$ and $N_{V_2-X}(u)  = N_{V_2-X}(v)$.
\end{enumerate}
\end{lem}
\begin{proof} By Lemma \ref{lem-blocksofQ}, $Q$ is a Gallai-tree, hence $Q-X$ is also a Gallai-tree.

If there exists a vertex $v \in V(Q)-X$ for which $|L^{\phi }(v)| > d_{G-X}(v)$, then for any extension $\psi$ of $\phi$ to $X \cup V_2$, $|L^{\psi }(v)| > d_{Q-X}(v)$, because each time a neighbor of $v$ in $V_2$ is coloured, its degree in the subgraph induced by uncoloured vertices decreases by 1, and its list decreases by at most 1. Hence $L^{\psi}|_{Q-X}$ is not a bad list assignment for $Q-X$, and $Q-X$ is $L^{\psi}$-colourable.

Assume
there are two non-cut vertices $u,v$ of the same block of $Q-X$ with $L^{\phi}(u) \ne L^{\phi}(v)$ and $N_{V_2-X}(u)  = N_{V_2-X}(v)$.
Assume   $\psi$ is an extension of $\phi$ to $X \cup V_2$. For any $w \in V_2$ adjacent to $u$ (hence also adjacent to $v$), either $\psi(w) \in L^{\phi}(u) \cap L^{\phi}(v)$, and or $\psi(w)$ is missing from ether $L^{\phi}(u) $ or $ L^{\phi}(v)$. If the latter occurs, then one of the following holds: \begin{itemize}
     \item  $|L^{\psi }(v)| > d_{Q-X}(v)$.
     \item $|L^{\psi }(u)| > d_{Q-X}(u)$.
  \end{itemize} If the latter case never occur, then  $L^{\psi}(u) \ne L^{\psi}(v)$. Thus, $L^{\psi}|_{Q-X}$ is not a bad list assignment for $Q-X$, and hence $Q-X$ is $L^{\psi}$-colourable.
\end{proof}

Since $G[V_2]$ is a planar graph and every planar graph with more than one vertex has at least two vertices of degree at most $5$, there exists an ordering $<$ of vertices in $V_2$ such that $v^* < v$ for every other vertex $v$ of $V_2$, and each vertex $u\in V_2$ has at most 5 neighbors $v$ with $v<u$.  
  
We colour vertices of $G$ one by one and use the following rules to choose the next vertex to be coloured. For $i=1,2,\ldots,$ let  $X_i$ be the set of the first $i$ coloured vertices, and let $\phi_i$ be the $L$-colouring of $G[X_i]$. Initially, $X_0 = \emptyset$. Assume for $i \ge 0$, $X_i$ and $\phi_i$ are defined. We construct $X_{i+1}$ and an extension $\phi_{i+1}$ of $\phi_i$ to $X_{i+1}$ as follows:  
 
\begin{enumerate}
    \item[(R1)]  If a  non-free connected component $Q$   has a vertex $v \in Q-X_i$ that is not a cut-vertex of $Q-X_i$, and $v$ is not adjacent to any vertex of $V_2-X_i$ and $|Q-X_{i}| \ge 2$, then let $$X_{i+1} = X_i \cup \{v\}$$ and let $\phi_{i+1}(v)$ be any colour in $L^{\phi_i}(v)$. 
    \item[(R2)] If (R1) does not apply, then let $u$ be the smallest vertex in $V_2-X_i$ with respect to the order $<$, and let 
    $$X_{i+1}  = X_i \cup \{u\}.$$ 
    \end{enumerate}

For each vertex $u \in V_2$, let $i_u$ be the index such that 
$$u \in X_{i_u} - X_{i_u-1}.$$
In other words, $u$ is the $i_u$-th coloured vertex.

The colour $\phi_{i+1}(u)$ will be selected carefully. The main task of the remaining part of the proof is to describe how to choose the colour for $u$ in this case.

Note that when (R1) is applied to colour $v \in Q-X_{i_u-1}$, since $|Q-X_{i_u-1}| \ge 2$, $|L^{\phi_{i_u-1}}(v) | \ge d_{G-X_{i_u-1}}(v) \ge 1$. So  $L^{\phi_{i_u-1}}(v)  \ne \emptyset$, and the required colour for $v$ exists.  Once a connected component $Q$ of $G[V_1]$ becomes free with respect to $\phi_i$, then it remains free with respect to $\phi_j$ for $j \ge i$. In the following, if  the partial colouring $\phi_i$ is clear from the context, we simply say $Q$ is free or non-free to mean that $Q$ is free or non-free with respect to $\phi_i$.

We apply the above rules until all vertices of $V_2$ are coloured. Our goal is to ensure that each connected component $Q$ of $G[V_1]$ is free when all vertices of $V_2$ are coloured.

If this goal is achieved, then $G$ has an $L$-colouring by the definition of free components. 
To complete the  proof of Theorem \ref{thm-main2}, it remains to show that this goal can indeed be achieved.

\begin{obs}
    \label{obs-2}
    Assume $u \in V_2$ and $Q$ is a non-free connected component of $G[V_1]$.
    Then $N_{Q-X_{i_u-1}}(u)= N_Q(u)  \ne \emptyset$. Moreover, if $|Q-X_{i_u-1}| \ge 2$, then each non-cut vertex  of $Q-X_{i_u-1}$ is adjacent to some vertex in $V(\theta_Q)-X_{i_u-1}$.
\end{obs}
\begin{proof} 
  Since $u$ is  uncoloured  (at step $i_u-1$), no vertex in $N_Q(u)$ has been coloured by (R1), i.e., $N_Q(u) \cap X_{i_u-1} = \emptyset$.  Since $Q$ is properly connected to $V_2$,  $N_{Q-X_{i_u-1}}(u) = N_Q(u) \ne \emptyset$.

   If $v$ is a non-cut vertex of $Q-X_{i_u-1}$ that is not adjacent to any vertex in $V(\theta_Q)-X_{i_u-1}$, then (R1) can be applied to colour $v$, which contradicts the definition of $i_u$.
\end{proof}

Let $F$ be a very nice subgraph   of $\Theta(G[V_2])$.

\begin{definition}
\label{def-protector}
 Assume $Q$ is a connected component of $G[V_1]$ and $u \in V(\theta_Q)$. Then   $u$ is a \textit{protector} of $Q$, unless one of the following holds:
 \begin{enumerate}
     \item $u\theta_Q \not\in E(F)$.
     \item $u=v^*$ and $v^*$ is not an isolated vertex of $G[V_2]$.
 \end{enumerate}
\end{definition}

It follows from the definition of very nice subgraph that each $u\in V_2$ is a protector of at most two connected components of $G[V_1]$. Moreover, $v^*$ is not a protector of any connected component of $G[V_1]$, 
each connected component $Q$ of $G[V_1]$ has at most two non-protectors, and $\theta^*$ has only one non-protector (namely $v^*$). Additionally, if  $w_1$ and $w_2$ are two non-protectors of $Q$, then there is a cycle $C_Q$   contained in $V(\theta_Q)$ such that  $w_1, w_2 \in V(C_Q)$ and $Q$ is contained in the interior of $C_Q$.

\begin{definition}
    \label{def-savior}
    Assume $Q$ is a connected component of $G[V_1]$  and $u \in V_2$ is a protector of $Q$.
    Let
    \begin{eqnarray*}
        S^*_{u,Q} &=& \{ c \in L^{\phi_{i_u-1}}(u): \exists \text{ an extension $\psi$ of $\phi_{i_u-1}$ to $X_{i_u-1} \cup V_2$ such that } \\ 
    & & \text{ $\psi(u)=c$ and $L^{\psi}|_{Q-X_{i_u-1}}$ is  a bad list assignment for $Q-X_{i_u-1}$}\}.
    \end{eqnarray*}
     If $|S^*_{u,Q}| \le 3$, then we say $u$ is a {\em savior} of $Q$. 
\end{definition}
Note that if $Q$ is free with respect to $\phi_{i_u-1}$, then $S^*_{u,Q} = \emptyset$ and hence $u$ is a savior for $Q$. 

To prove that $u$ is a savior for $Q$, it suffices to present a  set $S_{u,Q}$ of colours such that $S^*_{u,Q} \subseteq S_{u,Q}$ and $|S_{u,Q}| \le 3$. In other words, the set $S_{u,Q}$ satisfies 
Property (S): {\em for any extension $\psi$ of $\phi_{i_u-1}$ to $X_{i_u-1} \cup V_2 $, if $\psi(u) \notin S_{u,Q}$, then $\psi|_{Q-X_{i_u-1}}$ is not a bad list assignment for $Q-X_{i_u-1}$. }

In the following, for each savior $u$ of a connected component $Q$ of $G[V_1]$, a set $S_{u,Q}$ with Property (S) is given.
We call $S_{u,Q}$ the {\em colour cost set} for $u$ to be a savior of $Q$. 

Let $$\mathcal{Q}_u=\{Q: \text{ $Q$ is a connected component of $G[V_1]$  and $u$ is a savior of $Q$}\}.$$
Since $u$ is the protector of at most two connected components $Q$ of $G[V_1]$, $|\mathcal{Q}_u| \le 2$.

 
Now we can finish describing the colouring process by specifying the colour $\phi_{i_u}(u)$ for $u \in V_2$: 
$\phi_{i_u}(u)$ is any colour $c$ such that 
   $$c \in L^{\phi_{i_u-1}}(u) - \bigcup_{Q \in \mathcal{Q}_u}S_{u,Q}$$ 

Note that $\bigcup_{Q \in \mathcal{Q}_u}S_{u,Q}$ might be an empty set. In this case, $c$ can be any colour in $L^{\phi_{i_u-1}}(u)$. 

 It follows from the colouring rules that  $L^{\phi_{i_u-1}}(u)=L(u)-\{\phi_{i_u-1}(u'): u' \in  N_G(u) \cap  V_2, u' <u\}$. As $u$ has at most 5 neighbours $u' \in V_2$  for which $u' < u$, 
 $$|L^{\phi_{i_u-1}}(u)| = |L(u)| - |\{\phi_{i_u-1}(u'): u' \in  N_G(u) \cap V_2, u' <u\}|   \ge 12-5 = 7.$$
Since $|\mathcal{Q}_u| \le 2$ and $|S_{u,Q}| \le 3$ for each $Q \in \mathcal{Q}_u$, we conclude that 
$$L^{\phi_{i_u-1}}(u) - \bigcup_{Q \in \mathcal{Q}_u}S_{u,Q} \ne \emptyset,$$
and hence the required colour $c$ exists.

 \begin{lem}
     \label{lem-key}
     Each connected component $Q$ of $G[V_1]$ has a savior.
 \end{lem}

Assume Lemma \ref{lem-key} is true. We use the colouring strategy described above. When all vertices of $V_2$ are coloured, for any connected component $Q$ of $G[V_1]$, $Q$ is free. Let $X$ be the set of coloured vertices and $\phi$ be the partial colouring of $X$. Then $Q-X$   has a proper $L^{\phi}$-colouring. Therefore, $G$ has a proper $L$-colouring. 

It remains to prove 
 Lemma \ref{lem-key}.
 
\section{Proof of Lemma \ref{lem-key}}

  Assume Lemma \ref{lem-key} is not true and $Q$ is a connected component of $G[V_1]$ which has no savior. In particular, for any vertex $u \in V_2$, $Q$ is not free  at step $i_u$, and hence $|L^{\phi_{i_u-1}}(v)| = d_{G-X_{i_u-1}}(v)$ for each vertex $v \in Q-X_{i_u-1}$.



\begin{definition}
    \label{def-forced}
    Assume $u \le w$  are vertices in $V(\theta_Q)$ (where $u$ and $w$ are not necessarily distinct). We say $w$  is {\em confined   to colour $c$ at step $i_u$}    if  for any extension $\psi$ of $\phi_{i_u-1}$ to $X_{i_u-1}\cup V_2$ with $\psi(w) \ne c$, $L^{\psi}|_{Q-X_{i_u-1}}$ is not  bad   for $Q-X_{i_u-1}$.   We may simply say $w$ is confined at step $i_u$,  if the colour $c$ is clear from the context or is not important.   
\end{definition}

Since $u \le w$,   at step $i_u-1$, both $u$ and $w$ are uncoloured. If $w$ is confined at step $i_u$ to colour $c$, then at the step when $u$ is to be coloured, we may treat $w$ as coloured with colour $c$ (although $w$ is not coloured yet). Note that if $w$ is confined at step $i_u$ to colour $c$,   then $w$  is  confined at all later steps.  In particular, if there is a vertex $u$ such that $w$ is confined at step $i_u$ to colour $c$, then $w$ is confined to colour $c$ at step $i_w$.

\begin{lem}
    \label{lem-forcedc}
    Assume $w$ is confined at step $i_u$ to colour $c$. Then the following hold:
    \begin{enumerate}
        \item For any $v \in N_{Q-X_{i_u-1}}(w) $, $c \in L^{\phi_{i_u-1}}(v)$.
        \item $w$ is not a protector of $Q$.
    \end{enumerate}
\end{lem}
\begin{proof}
    (1) If $v \in N_{Q-X_{i_u-1}}(w) $ and $c \notin L^{\phi_{i_u-1}}(v)$, then for any extension $\psi$ of $\phi_{i_u-1}$ to $X_{i_u-1}\cup V_2$, either $\psi(w) \ne c$, or 
    $ \psi(w) = c \notin L^{\phi_{i_u-1}}(v)$ and hence $|L^{\psi}(v)| > d_{Q-X_{i_{u}-1}}(v)$. In any case, $Q-X_{i_u-1}$ is $L^{\psi}$-colourable. So  $Q$ is free at step $i_u-1$, a contradiction.   

    (2)  If $w$ is a protector of $Q$ and is confined at step $i_u$ to colour $c$, then $w$ is a savior of $Q$ with cost colour set $S_{w,Q}=\{c\}$, a contradiction.
\end{proof}

 \begin{lem}
     \label{lem-forced}
     Assume $u \le w \in V(\theta_Q)$, $v$ and $v'$ are non-cut vertices of a same block of  $Q-X_{i_u-1}$. If  $N_{V_2-X_{i_u-1}}(v) = N_{V_2-X_{i_u-1}}(v') \cup \{w\}$ and $c \in L^{\phi_{i_u-1}}(v)-L^{\phi_{i_u-1}}(v')$, then $w$ is confined at step $i_u$ to colour $c$. Consequently, $w$ is not a protector of $Q$. 
 \end{lem}
 \begin{proof}
     If $\psi$ is an extension of $\phi_{i_u-1}$ to $X_{i_u-1}\cup V_2$, and $\psi(w) \ne c$, then either $c \in L^{\psi}(v) - L^{\psi}(v')$ and hence $L^{\psi}(v) \ne L^{\psi}(v')$, or  $v$ and $v'$ have a common neighbor $u'$  in $V_2-X_{i_u-1}$ which is coloured by $c \notin L^{\phi_{i_u-1}}(v')$ and hence $|L^{\psi}(v')| > d_{Q-X_{i_u-1}}(v')$.  In any case, it follows from Lemma \ref{lem-suff}   that  $Q-X_{i_u-1}$ is $L^{\psi}$-colourable. So $w$ is confined at step $i_u$ to colour $c$. By Lemma \ref{lem-forcedc},  $w$ is not a protector of $Q$.
 \end{proof}

Assume $u$ is a protector of $Q$. 
 Let $$F_u  = \{w \in V_2: w\text{  is a non-protector of } Q,  w\text{ is confined at step $i_u$ }  \}.$$

 \begin{lem}
     \label{lem-ww'}
     Assume $u$ is a protector of $Q$. 
 If $w,w' \in F_u$, and $w,w'$ are confined to a same colour $c$, then $N_{Q-X_{i_u-1}}(w) \cap N_{Q-X_{i_u-1}}(w') = \emptyset$.
 \end{lem}
 \begin{proof}
     If $w,w'$ are confined to a same colour $c$, and  $v \in N_{Q-X_{i_u-1}}(w) \cap N_{Q-X_{i_u-1}}(w')$, then  for any  extension $\psi$ of $\phi_{i_u-1}$ to $X_{i_u-1}\cup V_2$, either $\psi(w) \ne c$ or $\psi(w') \ne c$, or $|L^{\psi}(v)| > d_{Q-X_{i_u-1}}(v) $. In any case, $Q-X_{i_u-1}$ is $L^{\psi}$-colourable. So $Q$ is free,  a contradiction.
 \end{proof}
 
\begin{lem}
    \label{lem-fu}
 Assume $u$ is a protector of $Q$ and  $v \in N_{Q-X_{i_u-1}}(u)$. 
Then $d_{G-X_{i_u-1}-F_u}(v)  \ge 4$. 
\end{lem}
  \begin{proof}
        Assume to the contrary that $d_{G-X_{i_u-1}-F_u}(v)  \le 3$. For $w_i\in F_u$, assume $w_i$ is confined at step $i_u$ to colour $c_i$. Let $$S_{u,Q} = L^{\phi_{i_u-1}}(v)- \{c_i: w_i \in F_u \text{ and } v\in N_{Q-X_{i_u-1}}(w_i)\}.$$ By Lemma \ref{lem-forcedc}, $\{c_i: w_i \in F_u,v\in N_{Q-X_{i_u-1}}(w_i)\} \subseteq L^{\phi_{i_u-1}}(v)$. By Lemma \ref{lem-ww'}, $|\{c_i: w_i \in F_u,v\in N_{Q-X_{i_u-1}}(w_i)\}| = |\{w_i \in F_u:v\in N_{Q-X_{i_u-1}}(w_i)\}|$. Since $|L^{\phi_{i_u-1}}(v)| = d_{G-X_{i_u-1}}(v)$, we have  $|S_{u,Q}| = d_{G-X_{i_u-1}-F_u}(v)  \le 3$.  Let $\psi$ be an extension of $\phi_{i_u-1}$ to $X_{i_u-1}\cup V_2$ with $\psi(u) \notin S_{u,Q}$. If   $\psi(w_i) \ne  c_i$ for  some $w_i \in F_u$, then $Q-X_{i_u-1}$ is $L^{\psi}$-colourable.   
        
        If $\psi(w_i) =  c_i$ for each $w_i \in F_u$, then as $\psi(u) \notin S_{u,Q}$,  $|L^{\psi}(v)| > d_{Q-X_{i_u-1}}(v)$, and hence $L^{\psi}|_{Q-X_{i_u-1}}$ is  not bad for $Q-X_{i_u-1}$. Therefore $u$ is a savior for $Q$ with cost colour set $S_{u,Q}$, a contradiction.
  \end{proof}

\begin{lem}
 \label{lem-1}   
 $|V_2|\ge 2$. 
\end{lem}

\begin{proof}
Assume $ V_2 =\{v^*\}$. Then $|L(v^*)|=12$. As $G[V_2]$ has a single face,     $G[V_1]$ is connected.  

By Observation \ref{obs-2}, each non-cut vertex of $G[V_1]-X_{i_{v^*}-1}$ is adjacent to $v^*$. 
If $B$ is a leaf block of $G[V_1]-X_{i_{v^*}-1}$ and $v\in U_B$, then 
by Lemma \ref{lem-blocksofQ},  $d_B(v)\le 2$ and thus $d_{G}(v)\le 3$, in contrary to Lemma \ref{lem-fu}.
\end{proof}

 \begin{lem}
     \label{lem-ge2}
    Assume $u$ is a protector of $Q$. Then $|Q-X_{i_u-1}| \ge 2$. 
 \end{lem}
 \begin{proof}
     Assume $Q-X_{i_u-1}$ consists of a single vertex $v$ for some protector $u$ of $Q$.  By Observation \ref{obs-2}, $u$ is adjacent to $v$. Since $Q$ has only  two non-protectors, we have $d_{G-X_{i_u-1}}(v)=d_{V_2-X_{i_u-1}}(v) \le 3$, which contradicts Lemma \ref{lem-fu}.
 \end{proof}

\begin{lem}
 \label{lem-not out face}  
$\theta_Q$ is a finite face of $G[V_2]$.
\end{lem}
\begin{proof}
Assume $\theta_Q$ is the infinite face of $G[V_2]$ and $v^*\in  V(\theta_Q)$. Thus, each vertex in $V(\theta_Q)-\{v^*\}$ is a protector of $Q$. Let $u$ be the last protector of $Q$.  As $v^*$ is the first vertex in the ordering, by Lemma \ref{lem-1}, $u \ne v^*$ and  $v^*\in X_{i_u-1}$. By Observation \ref{obs-2}, each non-cut vertex of $Q-X_{i_u-1}$ is adjacent to $u$. Assume $B$ is a leaf-block of $Q-X_{i_u-1}$ and $v\in U_B$. By Lemma \ref{lem-blocksofQ}, $d_B(v)\le 2$ and hence $d_{G-X_{i_u-1}}(v)\le 3$, contrary to Lemma \ref{lem-fu}.   
\end{proof}

By Lemma \ref{lem-not out face} and Corollary \ref{cor-C=theta}, we conclude that $\theta_Q=C_Q$ and for any block $B$ of $Q$, there are three  paths contained in $V(\theta_{Q}) \cup V(Q)$ connecting $V(\theta_{Q})$ and $V(B)$, and these paths are pairwise vertex disjoint except that they may have the same end vertex in $B$.

Let $w_1,w_2$ be the non-protectors of $Q$.  Since there are three paths connecting  $V(\theta_Q)$ and $Q$, and these paths have  different end vertices in $V(\theta_Q)$, there is at least one protector of $Q$ adjacent to a vertex in $Q$.


  \begin{lem}
       \label{lem-nonprotector}
      Assume $u$ is a protector of $Q$. Then each non-cut vertex $v$ of $Q-X_{i_u-1}$ is adjacent to at least one non-protector of $Q$. 
   \end{lem}
   \begin{proof}
        Assume $v$ is a non-cut vertex of $Q-X_{i_u-1}$ that is not adjacent to any non-protector of $Q$.  
        By Observation \ref{obs-2}, $v$ is adjacent to some vertex of $V_2-X_{i_u-1}$. Thus, $v$ is adjacent to a protector of $Q$. Let $u$ be the last protector of $Q$ which is adjacent to $v$. By Lemma \ref{lem-blocksofQ}, $d_{Q-X_{i_u-1}}(v) \le 2$. Hence, 
       $d_{G-X_{i_u-1}}(v) \le 3$, which contradicts Lemma \ref{lem-fu}.
\end{proof}


\begin{lem} \label{lem-leaf}
    If $B$ is a leaf-block of $Q$, then $w_1, w_2 \in N_{\theta_Q}(U_B)$. 
\end{lem}
\begin{proof}
    Assume to the contrary that $w_2 \notin N_{\theta_Q}(U_B)$.  By Corollary \ref{cor-C=theta},   $|N_{\theta_Q}(U_B)| \ge 2$. Thus,  $N_{\theta_Q}(U_B)$ contains a protector of $Q$. Let $u$ be the last protector of $Q$   that is adjacent to some vertex in $U_B$.
    
    By Lemma \ref{lem-fu}, $d_{G-X_{i_u-1}}(v)= |L^{\phi_{i_u-1}}(v)| \ge 4$ for each vertex $v\in N_{Q-X_{i_u-1}}(u)\cap U_B$.

Let $v_1 \in N_{Q-X_{i_u-1}}(u) \cap U_B$. As $N_{V_2-X_{i_u-1}}(v_1)  \subseteq \{u, w_1\}$,   it follows from Lemma \ref{lem-blocksofQ} that $B-X_{i_u-1}$ is an odd cycle $C=[v_1v_2\ldots v_{2l+1}]$. So   $d_{Q-X_{i_u-1}}(v_1)=2$,   $d_{G-X_{i_u-1}}(v_1)=4$,  $N_{V_2-X_{i_u-1}}(v_1)= \{u, w_1\}$. As $C$ contains at most one cut-vertex of $Q-X_{i_u-1}$, we may assume that $v_2 \in U_B$.   By assumption, $N_{V_2-X_{i_u-1}}(v_2) \subseteq \{u,w_1\}$. 
    By Lemma \ref{lem-nonprotector}, $v_2$ is adjacent to $w_1$. 

    If $v_2$ is not adjacent to $u$, then $N_{G-X_{i_u-1}}(v_2)=N_{G-X_{i_u-1}}(v_1) - \{u\}$,  and hence $u$ is confined at step $i_u$ to   colour $c \in L^{\phi_{i_u-1}}(v_1)- L^{\phi_{i_u-1}}(v_2)$, contrary to Lemma \ref{lem-forced}.

    Assume $v_2$ is also adjacent to $u$. Then $Q-\{v_1,v_2\}$ is contained in the interior of the 4-cycle $[uv_1w_1v_2]$. 
    If $Q$ has another leaf-block $B'$, then $|N_{\theta_Q}(U_{B'}) | \le 1 $,  a contradiction to Corollary \ref{cor-C=theta}. Thus, $Q-X_{i_u-1}$ is an odd cycle. 
    By  Lemma \ref{lem-nonprotector}, $v_{2l+1}$ is adjacent to $w_1$.  Therefore, $v_{2l+1}$ is not adjacent to $u$ (for otherwise $G$ contains $K_5$ as a minor). Hence,   $N_{G-X_{i_u-1}}(v_{2l+1})=N_{G-X_{i_u-1}}(v_1) - \{u\}$,  and thus $u$ is confined at step $i_u$ to colour the $c \in L^{\phi_{i_u-1}}(v_1)- L^{\phi_{i_u-1}}(v_{2l+1})$, contrary to Lemma \ref{lem-forced}. 
\end{proof}

The following   corollary follows from Lemma \ref{lem-leaf} and the planarity of $G$.

\begin{cor}
    \label{cor2}
  $Q$ has at most two leaf-blocks, and each protector $u$ of $Q$ in $V(\theta_Q)$ is adjacent to non-root vertices of exactly one leaf-block of $Q$. 
\end{cor}
\begin{proof}
    If $Q$ has three leaf-blocks, then $Q$ contains $K_{1,3}$ as a minor and by Lemma \ref{lem-leaf}, $Q \cup \{w_1,w_2\}$ contains $K_{3,3}$ as a minor, a contradiction.

    Assume $Q$ has two leaf-blocks $B$ and $B'$, and  $u$ is a protector of $Q$ adjacent to both $U_B$ and $U_{B'}$. 
    By contracting  the cycle $\theta_Q$ into a triangle containing $w_1,w_2,u$, and contracting each of $U_B$ and $U_{B'}$ into a single vertex, we obtain  a copy of $K_5$,   a contradiction. 
\end{proof}

It follows from Corollary \ref{cor2}  that  the blocks of $Q$  are ordered as $B_1,B_2, \ldots, B_k$ ($k \ge 1$) such that $B_1$ and $B_k$ are leaf-blocks and for $i=1,2, \ldots, k-1$, $B_i$ and $B_{i+1}$ share a cut-vertex $z_i$, and $z_1, z_2, \ldots, z_{k-1}$ are pairwise distinct.  

In the remainder of this section, let $u$ be the last protector of $Q$. 

By Corollary \ref{cor2}, we may assume that $N_Q(u) \subseteq U_{B_k}$, and assume $v_1 \in N_Q(u) \cap U_{B_k}$.

\begin{lem}
    \label{lem-w1w2}
    $w_1, w_2 \notin X_{i_u-1}$ and hence $(N_Q(w_1) \cup N_Q(w_2) ) \cap X_{i_u-1} = \emptyset$.
\end{lem}
\begin{proof}
    Assume to the contrary that $w_2 \in X_{i_u-1}$. By Lemma \ref{lem-blocksofQ}, $d_{Q-X_{i_u-1}}(v_1) \le 2$. By Lemma \ref{lem-fu}, $d_{G-X_{i_u-1}}(v_1) \ge 4$. It follows that $d_{Q-X_{i_u-1}}(v_1) = 2$, $N_{V_2-X_{i_u-1}}(v_1) = \{u, w_1\}$, $|L^{\phi_{i_u-1}}(v_1)| = d_{G-X_{i_u-1}}(v_1) = 4$ and
 $B_k - X_{i_u-1}$ is an odd cycle $[v_1v_2\ldots v_{2l+1}]$. 

Assume first that  $U_{B_k-X_{i_u-1}} - N_G(u) \ne \emptyset$, say $v'\in U_{B_k-X_{i_u-1}} - N_G(u)$. By Lemma \ref{lem-nonprotector}, $N_{V_2-X_{i_u-1}}(v') = \{w_1\}$ and $|L^{\phi_{i_u-1}}(v')| = d_{G-X_{i_u-1}}(v') = 3$. Let $c \in L^{\phi_{i_u-1}}(v_1) - L^{\phi_{i_u-1}}(v')$. If $\psi$ is an  extension   of $\phi_{i_u-1}$ to $V_2 \cup X_{i_u-1}$ with $\psi(u) \ne c$, then $L^{\psi}(v') \ne L^{\psi}(v_1)$ or $|L^{\psi}(v')| > d_{Q-X_{i_u-1}}(v')$. Thus, $Q-X_{i_u-1}$ is $L^{\psi}$-colourable. Therefore $u$ is confined to colour $c$ at step $i_u$, a contradiction to Lemma \ref{lem-forced}. 

Thus $U_{B_k-X_{i_u-1}} \subseteq  N_G(u)$. Since $u$ has at most two neigbours in $U_{B_k-X_{i_u-1}}$ (for otherwise $u$ has a neighbour $v_i$ with $d_{G-X_{i_u-1}}(v_i) =3$, contrary to Lemma \ref{lem-fu}), we conclude that $B_k-X_{i_u-1}$ is a triangle $[v_1v_2v_3]$ and $|U_{B_k-X_{i_u-1}}|=2$. Suppose $v_3$ is a cut-vertex of $G$, and $v_1,v_2 \in N_G(u) \cap N_G(w_1)$. This implies that $Q$ has at least two blocks, i.e., $k \ge 2$. But then $Q - \{v_1, v_2\}$ is contained in the interior of the 4-cycle $[w_1v_1uv_2]$. Hence $w_2$ is not adjacent to $U_{B_1}$ in $G$, which contradics Lemma \ref{lem-leaf}. Therefore $w_1,w_2 \notin X_{i_u-1}$. By Rule (R1), no neighbour  of $w_1,w_2$ in $Q$ is coloured at step $i_u$. Hence $(N_Q(w_1) \cup N_Q(w_2)) \cap X_{i_u-1} = \emptyset$.
\end{proof}

 By Lemma \ref{lem-blocksofQ} and Lemma \ref{lem-nonprotector}, each non-cut vertex $v$ of $Q$ is adjacent to at least one of $w_1,w_2$. As $w_1,w_2\notin X_{i_u-1}$, we have  $V(Q) \cap X_{i_u-1} = \emptyset$ and $Q-X_{i_u-1}=Q$.

\begin{lem}
    \label{lem-twoleaves}
    $Q$ has at least two blocks. 
\end{lem}
\begin{proof}
  Assume to the contrary that $Q$ has a single block. 

If $Q$ is a copy of $K_2$ with vertices $v_1,v_2$, then $d_Q(v_i) =1$ and $d_{G-X_{i_u-1}}(v_i) \le 4$ for $i=1,2$. 
If $d_{G-X_{i_u-1}}(v_1) = d_{G-X_{i_u-1}}(v_2)=4$, then $\{u, w_1, w_2\} \subseteq N_{V_2}(v_1), N_{V_2}(v_2)$. Together with $V(\theta_Q)$, we obtain $K_5$ as a minor, which is a contradiction. Thus we may assume that $d_{G-X_{i_u-1}}(v_2) \le 3$.  Hence
$v_1$ is adjacent to $u,w_1,w_2$, and by Lemma \ref{lem-fu}, $v_2$ is not adjacent to $u$. 
If $L^{\phi_{i_u-1}}(v_2) \not\subseteq L^{\phi_{i_u-1}}(v_1)$, then 
for any extension $\psi$ of $\phi_{i_u-1}$ to $X_{i_u-1}\cup V_2$, either $|L^{\psi}(v_1)| > d_{Q}(v_1)$ or $L^{\psi}(v_1) \not=  L^{\psi}(v_2)$,
and hence  $Q$ is $L^{\psi}$-colourable. Thus $Q$ is free, a contradiction.

Therefore we may assume that $L^{\phi_{i_u-1}}(v_2)  \subseteq L^{\phi_{i_u-1}}(v_1)$. 
Let $$S_{u,Q} = L^{\phi_{i_u-1}}(v_1) - L^{\phi_{i_u-1}}(v_2).$$ 
Then $|S_{u,Q}| \le 3$. Let $\psi$ be an extension of $\phi_{i_u-1}$ to $V_2 \cup X_{i_u-1}$ for which $\psi(u) \notin S_{u,Q}$. We shall show that $Q$ is $L^{\psi}$-colourable, and hence $u$ is a savior for $Q$ with cost colour set $S_{u,Q}$. 

Assume to the contrary that $L^{\psi}|_{Q}$ is  bad for $Q$.
Note that $|S_{u,Q}| = |L^{\phi_{i_u-1}}(v_1)| - |L^{\phi_{i_u-1}}(v_2)| = |\{w_1,w_2\} - N_G(v_2)|+1$. If $w_i$ is adjacent to $v_2$, then $\psi(w_i) \in L^{\phi_{i_u-1}}(v_2)$, for otherwise $|L^{\psi}(v_2)| >  d_{Q}(v_2)$.  Hence $\psi(w_i)  \notin S_{u,Q}$.
So $S_{u,Q} - \{\psi(w_1), \psi(w_2)\} \ne \emptyset$, which implies that $L^{\psi}(v_1) \ne L^{\psi}(v_2)$, a contradiction.

Assume $Q$ is an odd cycle $C=[v_1v_2\ldots v_{2l+1}]$. Note that for each vertex $x \in \{u, w_1, w_2\}$, $N_{Q}(x)$ is a subpath of $C$, and $u$ has at most two neighbours in $C$, for otherwise $u$ has a neighbour $v_i$ with $d_{G-X_{i_u-1}}(v_i) =3$, contrary to Lemma \ref{lem-fu}.

\medskip
{\bf Case 1}  $u$ has two neighbours in $C$.
\medskip

Assume $u$ is adjacent to  $v_1$ and $v_2$.    
By Lemma \ref{lem-nonprotector}, for $i \in \{1,2,\ldots, 2l+1\}$,  $v_i$ is adjacent to at least one non-protector.  On the other hand, at least one of $v_1, v_2$ is adjacent to only one non-protector, for otherwise $G$ contains a subdivision of $K_5$. By symmetry, we assume that $N_{V_2-X_{i_u-1}}(v_2)= \{u, w_1\}$. 
If  $N_{V_2-X_{i_u-1}}(v_3)= \{  w_1\}$, then by Lemma \ref{lem-forced}, $u$ is confined at step $i_u$  to colour $c \in L^{\phi_{i_u-1}}(v_2)- L^{\phi_{i_u-1}}(v_3)$, contrary to Lemma \ref{lem-forcedc}. If $N_{V_2-X_{i_u-1}}(v_3)= \{  w_2\}$, then 
$N_{V_2-X_{i_u-1}}(v_i)= \{  w_2\}$ for $i=3,4,\ldots, 2l+1$ and $N_{V_2-X_{i_u-1}}(v_1)= \{ u, w_2\}$. Again by Lemma \ref{lem-forced}, $u$ is confined at step $i_u$ to colour $c \in L^{\phi_{i_u-1}}(v_1)- L^{\phi_{i_u-1}}(v_{2l+1})$, contrary to Lemma \ref{lem-forcedc}. Thus $N_{V_2-X_{i_u-1}}(v_3)= \{  w_1, w_2\}$. This implies that $N_{V_2-X_{i_u-1}}(v_1)= \{  u, w_2\}$. By symmetry, $N_{V_2-X_{i_u-1}}(v_{2l+1})= \{  w_1, w_2\}$, and this implies that $l=1$, i.e.,  
$C=[v_1v_2v_3]$.

 If $L^{\phi_{i_u-1}}(v_1) = L^{\phi_{i_u-1}}(v_2) $, then for any extension $\psi$ of $\phi_{i_u-1}$ to $X_{i_u-1}\cup V_2$, either $\psi(w_1) = \psi(w_2)$ and 
 $|L^{\psi}(v_3)| > d_{Q}(v_3)$, or $\psi(w_1) \ne \psi(w_2)$, and hence $L^{\psi}(v_1) \ne L^{\psi}(v_2)$. In any case, $Q$ is $L^{\psi}$-colourable, and $u$ is a savior for $Q$ with cost colour set $S_{u,Q}=\emptyset$. 
 
 If $ L^{\phi_{i_u-1}}(v_1) \ne L^{\phi_{i_u-1}}(v_2) $, then let $S_{u,Q} = L^{\phi_{i_u-1}}(v_1) \cap L^{\phi_{i_u-1}}(v_2)$. Then $|S_{u,Q}| \le 3$, and for any extension $\psi$ of $\phi_{i_u-1}$ to $V_2 \cup X_{i_u-1}$ with $\psi(u) \notin S_{u, Q}$, either $|L^{\psi}(v_1)| > d_{Q}(v_1)$ or $|L^{\psi}(v_2)| > d_{Q}(v_2)$. Thus,
 $Q$ is $L^{\psi}$-colourable, and $u$ is a savior for $Q$ with cost colour set $S_{u,Q}$.

\medskip
{\bf Case 2}  $u$ has only one neighbour in $C$.
\medskip

Assume $u$ is adjacent to $v_1$. If $v_1$ is adjacent to both $w_1$ and $w_2$, then there exists $i \in \{2, 2l+1\}$ such that $d_{G-X_{i_u-1}}(v_i)=3$. By symmetry, we may assume that $N_{V_2-X_{i_u-1}}(v_2) = \{ w_1\}$. Let $S_{u,Q}$ be a subset of $ L^{\phi_{i_u-1}}(v_1)-L^{\phi_{i_u-1}}(v_2)$ of size $2$.   For any extension $\psi$ of $\phi_{i_u-1}$ to $X_{i_u-1}\cup V_2$ for which $\psi(u)\notin S_{u,Q}$, at least one of the following holds:
\begin{itemize}
    \item $\psi(w_1)\notin L^{\phi_{i_u-1}}(v_2)$ and hence $|L^{\psi}(v_2)|>d_{Q}(v_2)$.
    \item $\psi(w_1)\in L^{\phi_{i_u-1}}(v_2)$ and hence  $L^{\psi}(v_1)\not= L^{\psi}(v_2)$, because at least one of the two colours in $S_{u,Q}$ is contained in $L^{\psi}(v_1) - L^{\psi}(v_2)$.
\end{itemize}
Thus, $Q$ is $L^{\psi}$-colourable, and $u$ is a savior for $Q$ with cost colour set $S_{u,Q}$.

Assume $v_1$ is adjacent to only one of $w_1, w_2$. By symmetry, assume that $N_{V_2-X_{i_u-1}}(v_1) = \{u, w_1\}$. 

 If $N_{V_2-X_{i_u-1}}(v_2)=\{w_1\}$, then it follows from Lemma \ref{lem-forced} that $u$ is confined at step $i_u$ to colour $c \in L^{\phi_{i_u-1}}(v_1) - L^{\phi_{i_u-1}}(v_2)$, a contradiction.  
 If  $N_{V_2-X_{i_u-1}}(v_2)=\{w_1, w_2\}$, then $w_2$ is adjacent to $v_2, v_3, \ldots, v_{2l+1}$, and $N_{V_2-X_{i_u-1}}(v_{2l+1})=\{w_2\}$. It follows from Lemma \ref{lem-forced} that $w_1$ is confined at step $i_u$ to colour $c \in L^{\phi_{i_u-1}}(v_2) - L^{\phi_{i_u-1}}(v_{2l+1})$. But then $d_{G-X_{i_u-1}-F_u}(v_1)=3$, contrary to Lemma \ref{lem-fu}. 

 Thus, $N_{V_2-X_{i_u-1}}(v_2)=\{w_2\}$. By symmetry, $N_{V_2-X_{i_u-1}}(v_{2l+1})=\{w_2\}$. Hence for $i=2,3,\ldots, 2l+1$,    $N_{V_2-X_{i_u-1}}(v_i)=\{w_2\}$.
  
Assume $w_1$ and $w_2$ are adjacent. 
If $L^{\phi_{i_u-1}}(v_2) \subseteq L^{\phi_{i_u-1}}(v_1) $,  then let $S_{u,Q} = L^{\phi_{i_u-1}}(v_2)$. Let  $\psi$ be an extension of $\phi_{i_u-1}$ to $X_{i_u-1}\cup V_2$ for which $\psi(u)\notin S_{u,Q}$. Since $\psi(w_1) \ne \psi(w_2)$,  either  ${\psi}(u) \in L^{\phi_{i_u-1}}(v_1)-L^{\phi_{i_u-1}}(v_2)$ and hence $L^{\psi}(v_1) \ne L^{\psi}(v_2)$ or ${\psi}(u)\notin L^{\phi_{i_u-1}}(v_1)$ and hence $|L^{\psi}(v_1)|>d_{Q}(v_1)$. In any case, $Q$ is $L^{\psi}$-colourable, and $u$ is a savior for $Q$ with cost colour set $S_{u,Q}$.

If $L^{\phi_{i_u-1}}(v_2) \not\subseteq L^{\phi_{i_u-1}}(v_1) $, then $|L^{\phi_{i_u-1}}(v_1) - L^{\phi_{i_u-1}}(v_2)|\ge 2 $. Let $S_{u,Q}$ be a subset of $L^{\phi_{i_u-1}}(v_1)-L^{\phi_{i_u-1}}(v_2)$ of size $2$. For any extension $\psi$ of $\phi_{i_u-1}$ to $X_{i_u-1}\cup V_2$ for which $\psi(u)\notin S_{u,Q}$, at least one of the colours in $S_{u,Q}$ is contained in $L^{\psi}(v_1) - L^{\psi}(v_2)$. Hence 
$L^{\psi}(v_1) \ne L^{\psi}(v_2)$ and $Q$ is $L^{\psi}$-colourable. So $u$ is a savior for $Q$ with cost colour set $S_{u,Q}$.

Now assume $w_1$ and $w_2$ are not adjacent. Then $|\theta_Q| \ge 4$, and $\theta_Q$ contains another protector  of $Q$. As there are three pairwise vertex-disjoint paths connecting $\theta_Q$ to $v_2$ except that they have the same end vertex $v_2$, at most one of the three paths contains $w_2$ and at most one of the three paths contains $w_1$. Thus $Q$ has a  protector which is adjacent to $v_2$. By symmetry, $Q$ has a  protector which is adjacent to $v_{2l+1}$.
Let $w$ be the last protector of $Q$ which is adjacent to $v_2$ or $v_{2l+1}$. If $w$ is adjacent to both $v_2$ and $v_{2l+1}$, then $V(Q)\cup V(\theta_Q)$ contains $K_5$ as a minor, a contradiction. By symmetry, 
assume that $w$ is adjacent to $v_2$. Thus   $$N_{V_2-X_{i_w-1}}(v_{2l+1})  =  N_{V_2-X_{i_w-1}}(v_2)  -\{w\} .$$
This contradicts  Lemma \ref{lem-forced}. 
\end{proof}

It follows from Corollary \ref{cor2} and Lemma \ref{lem-twoleaves} that $k \ge 2$ and $Q$ has exactly two leaf-blocks $B_1$ and $B_k$. Recall that   $v_1 \in N_Q(u) \cap U_{B_k}$ and $N_Q(u) \subseteq U_{B_k}$. We complete the proof of Theorem \ref{thm-main2} by considering two cases. 



\medskip
\noindent
{\bf Case 1}  $|L^{\phi_{i_u-1}}(v_1)|=d_{G-X_{i_{u-1}}}(v_1)=5$.
\medskip

Then $B_k$ is an odd cycle $C=[v_1v_2\ldots v_{2l+1}]$  and $N_{V_2-X_{i_u-1}}(v_1) = \{u, w_1, w_2\}$.  By symmetry, we may assume that $v_2\in U_{B_k}$. 

By the planarity of $G$,  $v_2$ is adjacent to  only one of $u,w_1,w_2$. By Lemma \ref{lem-fu}, $v_2$ is not adjacent to $u$.  By symmetry, assume  $N_{V_2-X_{i_u-1}}(v_2)=\{w_1\}$.  Then $ |L^{\phi_{i_u-1}}(v_2)| = d_{G-X_{i_u-1}}(v_2)=3.$ Let $S_{u,Q}$ be a subset of $ L^{\phi_{i_u-1}}(v_1)-L^{\phi_{i_u-1}}(v_2)$ of size $2$.  For any extension $\psi$ of $\phi_{i_u-1}$ to $X_{i_u-1}\cup V_2$ for which $\psi(u)\notin S_{u,Q}$, at least one of the following holds:
\begin{itemize}
    \item $\psi(w_1)\notin L^{\phi_{i_u-1}}(v_2)$ and hence $|L^{\psi}(v_2)|>d_{Q}(v_2)$.
    \item $\psi(w_1)\in L^{\phi_{i_u-1}}(v_2)$ and hence $L^{\psi}(v_1)\not= L^{\psi}(v_2)$.
\end{itemize}

So $Q$ is $L^{\psi}$-colourable, and $u$ is a savior for $Q$ with cost colour set $S_{u,Q}$.

\medskip
\noindent
{\bf Case 2} $|L^{\phi_{i_u-1}}(v_1)|=d_{G-X_{i_{u-1}}}(v_1)=4$.
\medskip

 
First, we show that $B_1$ has a non-root vertex $v'_1$ with $d_{G-X_{i_u-1}}(v'_1) =3$.

If $B_1=K_2$, then $B_1$ has a single non-root vertex $v'_1$. By Lemma \ref{lem-leaf}, $v'_1$ is adjacent to both $w_1$ and $w_2$, and not adjacent to $u$.  Hence $d_{G-X_{i_u-1}}(v'_1) =3.$

If $B_1$ is an odd cycle $[v'_1v'_2 \ldots v'_{2t+1}]$, then we may assume that $v'_1,v'_2$ are non-root vertices of $B_1$, and at most one of $v'_1,v'_2$  is adjacent to both $w_1,w_2$ (for otherwise $V(\theta_Q)\cup V(Q)$ contains $K_5$ as a minor, a contradiction). Thus we may assume that   $|L^{\phi_{i_u-1}}(v'_1)|=d_{G-X_{i_u-1}}(v'_1) =3$.

 \medskip
 \noindent
 {\bf Case 2(i)} For every other non-root vertex $v'$ of $B_1$, 
 $L^{\phi_{i_u-1}}(v') = L^{\phi_{i_u-1}}(v'_1)$.
\medskip

Let $c \in L^{\phi_{i_u-1}}(v_1)- L^{\phi_{i_u-1}}(v'_1)$.

Let $l$ be the largest index such that $B_l$ has a vertex $x$ such that $c \notin  L^{\phi_{i_u-1}}(x)$. 
Since  $c \notin  L^{\phi_{i_u-1}}(v'_1)$, we know that $l \ge  1$ is well-defined. 
Let 
\[
S_{u, Q} = \begin{cases} \{c\}, &\text{ if $k-l$ is even}, \cr 
L^{\phi_{i_u-1}}(v_1)-\{c\}, &\text{ if $k-l$ is odd}.
\end{cases}
\]

Assume $\psi$ is an extension of $\phi_{i_u-1}$ to $X_{i_u-1} \cup V_2$ such that $\psi(u) \notin S_{u,Q}$, and $L^{\psi}|_{Q}$ is  a bad list assignment for $Q$. By Lemma \ref{lem-Gallai}, for $1 \le i \le q$, there is a set $C_i$ of colours such that for each vertex $x$ of $Q$, $L^{\psi}(x) = \cup_{x \in B_i}C_i$ and $C_i \cap C_{i+1} = \emptyset$.

For $i=1,2$, we have $\psi(w_i) \in L^{\phi_{i_u-1}}(v'_1)$, for otherwise since 
 $L^{\phi_{i_u-1}}(v') = L^{\phi_{i_u-1}}(v'_1)$ for every  non-root vertex $v'$ of $B_1$, and each $w_i$ is adjacent to some non-root vertex of $B_1$,  we conclude that 
 there is a non-root vertex $v'$ of $B_1$ such that 
$|L^{\psi}(v')| > d_{Q}(v')$, in contrary to the assumption that $L^{\psi}|_{Q}$ is a bad list assignment for $Q$. So $\psi(w_i)\ne c$ for $i=1,2$.

Since $w_1,w_2 \in N_G(U_{B_i})$ for $i=1,k$, the subgraph of $G$ induced by $\{w_1,w_2\} \cup U_{B_1} \cup U_{B_k}$ contains a cycle $C''$ such that 
$Q - U_{B_1}  - U_{B_k}$ is contained in the interior of $C''$, and $u$ is contained in the exterior of $C''$. Thus $u$ is not adjacent to any vertex in $Q - U_{B_1}  - U_{B_k}$. So for any vertex $x $ of $Q - U_{B_1}  - U_{B_k}$, $$L^{\phi_{i_u-1}}(x)  - \{\psi(w_1), \psi(w_2)\} \subseteq L^{\psi}(x).$$
In particular, $c \in L^{\psi}(x)$ if and only if $c \in L^{\phi_{i_u-1}}(x)$.

As there is a vertex $x \in B_l$ such that 
$c \notin L^{\phi_{i_u-1}}(x)$, we know that 
$c \notin C_l$.

As  $c \in L^{\phi_{i_u-1}}(x)$ for all $x \in B_{l+1}$, we know that $c \in L^{\psi}(x)$ for all $x \in B_{l+1}$. 
This implies that $c \in C_{l+1}$. As $C_{l+1} \cap C_{l+2} = \emptyset$, we know that $c \notin C_{l+2}$.  Now $c \in L^{\psi}(x)$ for all $x \in B_{l+3}$. 
This implies that $c \in C_{l+3}$.  Repeat this argument, we conclude that   $c \in C_{l+1}, C_{l+3}, \ldots, C_{l+1+2t}, \ldots$.  
If $k-l$ is odd, this implies that $c \in C_k$, and hence $\psi(u)\ne c$. Therefore $\psi(u) \notin L^{\phi_{i_u-1}}(v_1)$ and  $|L^{\psi}(v_1)| > d_{Q}(v_1)$.  So 
$L^{\psi}|_{Q}$ is not a bad list assignment for $Q$, a contradiction.

If $k-l$ is even, then  $c \notin C_l, C_{l+2}, \ldots, C_k$. But $\psi(u) \notin S_{u,Q} = \{c\}, \psi(w_i) \ne c$ for $i=1,2$. Hence $c \in L^{\psi}(v_1)$, a contradiction.

\medskip
\noindent
{\bf Case 2(ii)} 
 $L^{\phi_{i_u-1}}(v'_1) \ne L^{\phi_{i_u-1}}(v'_2)$ for   $v'_1, v'_2 \in U_{B_1}$.
 \medskip
 
In this case  $B_1 $ is an odd cycle $C'=[v'_1v'_2\ldots v'_{2t+1}]$. Note that $N_{V_2-X_{i_u-1}}(v'_1) \ne N_{V_2-X_{i_u-1}}(v'_2)$, for otherwise for any extension $\psi$ of $\phi_{i_u-1}$ to $X_{i_u-1}\cup V_2$, we have either $L^{\psi}(v'_1) \ne L^{\psi}(v'_2)$ or $|L^{\psi}(v'_i)| > d_{Q}(v'_i)$ for some $i \in \{1,2\}$, and hence $u$ is a savior of $Q$ with cost colour set $S_{u,Q}= \emptyset$. 

If one of the $v'_1,v'_2$ is adjacent to both $w_1$ and $w_2$, without loss of generality, we assume $N_{V_2-X_{i_u-1}}(v'_2)=\{w_1,w_2\}$ and $N_{V_2-X_{i_u-1}}(v'_1) = \{w_2\}$. By Lemma \ref{lem-forced}, $w_1$ is confined at step $i_u$ to colour $c\in L^{\phi_{i_u-1}}(v'_2)-L^{\phi_{i_u-1}}(v'_1)$. 

If $B_k$ has a non-root vertex $v$ which is adjacent to $u$ and $w_1$, then $d_{G-X_{i_u-1}-F_u}(v)=3$, contrary to Lemma \ref{lem-fu}.

If each non-root vertex of $B_k$ is not adjacent to both  $u$ and $w_1$, then $B_k$ is an odd cycle $[v_1v_2\ldots v_{2l+1}]$.  Since $v_1\in N_Q(u)\cap U_{B_k}$, we assume that $N_{V_2-X_{i_u-1}}(v_1)=\{u,w_2\}$. 
Since $w_1,w_2\in N_{V_2}(U_{B_k})$, there exists a  vertex $v_i\in U_{B_k}$ such that $N_{V_2-X_{i_u-1}}(v_i)=\{w_1\}$.
It follows from Lemma \ref{lem-forcedc} that $c \in L^{\phi_{i_u-1}}(v_i)$. Then let $S_{u,Q}$ be a subset of $L^{\phi_{i_u-1}}(v_1)-(L^{\phi_{i_u-1}}(v_i)-\{c\})$ of size $2$.

For any extension $\psi$ of $\phi_{i_u-1}$ to $X_{i_u-1}\cup V_2$ for which $\psi(u)\notin S_{u,Q}$,  $\psi (w_1)\not=c$ or $\psi (w_1)=c$ but then at least one of the colours in $S_{u,Q}$ is contained in $L^{\psi}(v_1) - L^{\psi}(v_i)$ and hence 
$L^{\psi}(v_1) \ne L^{\psi}(v_i)$. In any case $Q$ is $L^{\psi}$-colourable. So $u$ is a savior for $Q$ with cost colour set $S_{u,Q}$.


Assume each of $v'_1$ and $v'_2$ is adjacent to exactly one non-protector, say  $N_{V_2-X_{i_u-1}}(v'_1)= \{w_2\}$ and $ N_{V_2-X_{i_u-1}}(v'_2)= \{w_1\}$. Assume $\psi$ is an extension  of $\phi_{i_u-1}$ to $X_{i_u-1}\cup V_i$ such that $L^{\psi}|_{Q}$ is bad for $Q$. As $L^{\phi_{i_u-1}}(v'_1) \ne L^{\phi_{i_u-1}}(v'_2)$ and $L^{\psi}(v'_1) = L^{\psi}(v'_2)$,
we conclude that $w_1$ is confined at step $i_u$ to 
the unique colour $c_1 \in L^{\psi}(v'_1) - L^{\psi}(v'_2)$ 
and $w_2$ is confined at step $i_u$ to the unique colour $c_2 \in L^{\psi}(v'_2) - L^{\psi}(v'_1)$. Thus $w_1, w_2 \in F_u$, and hence $d_{G-X_{i_u-1}-F_u}(v_1)  \le 3$, contrary to Lemma \ref{lem-fu}. This completes the proof of Lemma \ref{lem-key}. \qed

	\end{document}